\documentclass[a4paper,11pt,oneside]{amsart}
\usepackage[DIV=14,BCOR=2mm,headinclude=false,footinclude=false]{typearea}
\usepackage{pgf, tikz}
\usepackage{pgfplots}
\usetikzlibrary{shapes.geometric}
\usetikzlibrary{arrows.meta,arrows}
\usepackage[latin1]{inputenc}
\usepackage{amsfonts}
\usepackage{amsmath}
\usepackage{amsthm}
\usepackage{amssymb}
\usepackage{mathtools}
\usepackage{amsbsy, hyperref}
\usepackage{mathrsfs}
\usepackage{mathabx}
\usepackage{dsfont}
\usepackage{tcolorbox}
\usepackage{comment}
\usepackage{bbm}
\usepackage{enumerate}
\usepackage[numbers]{natbib}
\usepackage{graphicx,color}
\usepackage{yfonts}
\usepackage[labelfont=rm,format=plain,indention=0cm,singlelinecheck=off,justification=raggedright,skip=2pt]{caption}
\usepackage[fit]{truncate}
\usepackage{placeins} 
\usepackage{flafter}  
\usepackage{float}

 \usepackage{todonotes}

\usepackage[framemethod=tikz]{mdframed}
\mdfsetup{linecolor=blue,backgroundcolor=gray!10,roundcorner=10pt,leftmargin=2pt,innerleftmargin=2pt,innerrightmargin=2pt}

\newcommand{\hh}{\mathfrak{h}}

\newcommand{\PP}{\mathbb{P}}
\newcommand{\QQ}{\mathbb{Q}}

\newcommand{\FF}{\mathcal{F}}

\newcommand{\CC}{\mathrm{C}}
\newcommand{\RR}{\mathrm{R}}
\newcommand{\TT}{\mathrm{T}}

\newcommand{\size}{\mathrm{Size\, }}
\newcommand{\ssize}{\widetilde{\mathrm{size}\,}}
\newcommand{\energy}{\mathrm{Energy \,}}

\newcommand{\BHT}{{\textrm{BHT}}}

\def\Xint#1{\mathchoice
   {\XXint\displaystyle\textstyle{#1}}%
   {\XXint\textstyle\scriptstyle{#1}}%
   {\XXint\scriptstyle\scriptscriptstyle{#1}}%
   {\XXint\scriptscriptstyle\scriptscriptstyle{#1}}%
   \!\int}
\def\XXint#1#2#3{{\setbox0=\hbox{$#1{#2#3}{\int}$}
     \vcenter{\hbox{$#2#3$}}\kern-.5\wd0}}

\def\aver#1{\Xint-_{#1}}

\usetikzlibrary{calc,matrix,backgrounds}

\pgfdeclarelayer{overlay}
\pgfsetlayers{overlay,background,main}

\tikzset{circle/.style = {rounded corners,line width=1bp,color=#1}}%

\DeclareSymbolFont{yhlargesymbols}{OMX}{yhex}{m}{n} 
\DeclareMathAccent{\yhwidehat}{\mathord}{yhlargesymbols}{"62}

\makeatletter

\def\Xint#1{\mathchoice
   {\XXint\displaystyle\textstyle{#1}}%
   {\XXint\textstyle\scriptstyle{#1}}%
   {\XXint\scriptstyle\scriptscriptstyle{#1}}%
   {\XXint\scriptscriptstyle\scriptscriptstyle{#1}}%
   \!\int}
\def\XXint#1#2#3{{\setbox0=\hbox{$#1{#2#3}{\int}$}
     \vcenter{\hbox{$#2#3$}}\kern-.5\wd0}}

\def\aver#1{\Xint-_{#1}}

\theoremstyle{plain}
\newtheorem{theorem}{Theorem}[section]

\newtheorem{lemma}[theorem]{Lemma}

\newtheorem{proposition}[theorem]{Proposition}

\newtheorem{conjecture}[theorem]{Conjecture}
\theoremstyle{definition}
\newtheorem{definition}[theorem]{Definition}

\newtheorem{remark}[theorem]{Remark}

\title[On a bilinear square function built on a collection of strips]{On the boundedness for the bilinear quadratic functional given by arbitrary strips}

\author[F. Bernicot]{Fr\'ed\'eric Bernicot}
\address{CNRS - Nantes Universit\'e \\ Laboratoire de Math\'ematiques Jean Leray \\ 2, Rue de la Houssini\`ere F-44322 Nantes Cedex 03, France}
\email{frederic.bernicot@univ-nantes.fr}

\date{\today}


\begin{document}

\begin{abstract} This study provides initial results on the boundedness of the (smooth) bilinear quadratic functional defined by an arbitrary collection of disjoint strips. The square function under consideration is a combination between the well-known Rubio de Francia square function from the linear setting with the bilinear Hilbert transform's singularity structure, which involves modulation invariance.
\end{abstract}

\maketitle

\tableofcontents

\section{Introduction}

In this work, we aim to give a first result concerning the boundedness of the bilinear square function, built on the $\BHT$ (bilinear Hilbert transform) structure and an arbitrary collection of strips. 

Let us first review this topic which originates to the work of Rubio de Francia \cite{RF}: consider $\Omega:=(\omega)_{\omega\in\Omega}$ an arbitrary collection of disjoint open intervals of the real line, then the associated (sub)linear square function
$$ RF_{\Omega}(f):= \left(\sum_{\omega \in \Omega} |\pi_{\omega}(f)|^2 \right)^{1\over 2}$$
is bounded in $L^p({\mathbb R})$ for every $p\in [2,\infty)$,
where here $\pi_{\omega}$ stands for the non-smooth frequency projection on $\omega$
$$ \pi_{\omega}(f):= x \mapsto \int e^{ix\xi} {\bf 1}_{\omega}(\xi) \hat{f}(\xi) d\xi.$$ This important result encodes the orthogonality of the projections $\pi_{\omega}$ due to the disjointness in frequency, in any Lesbesgue spaces $L^p$, $p\in[2,\infty)$. If for some specific configurations (for example the dyadic intervals $\omega_k:=(2^k,2^{k+1})$) one can enlarge the range of boundedness, it is known that for the general statement the range $[2,\infty)$ is the maximal one.

We refer the reader to \cite{BB} for more details about the literature on this topic and to its bilinear version. In fact, if on the real line, the frequency domain has a simple geometry (described by a collection of intervals), in the bilinear context the frequency domain is the plane ${\mathbb R}^2$ and there could be far more subtle geometric aspects (curvature, bi-parameter, invariance by non-trivial modulation, ...).

So now, consider ${\mathcal U}:=(U)_{U\in{\mathcal U}}$ a collection of disjoint arbitrary open subsets of ${\mathbb R}^2$, we can consider the bilinear quadratic functional 
$$ RF_{{\mathcal U}}(f,g):= \left(\sum_{\omega \in \Omega} |\pi_{U}(f,g)|^2 \right)^{1\over 2}$$
where now $\pi_{U}$ is the bilinear projection
$$ \pi_{U}(f,g):= x \mapsto \int e^{ix(\xi+\eta)} {\bf 1}_{U}(\xi,\eta) \hat{f}(\xi) \hat{g}(\eta) d\xi d\eta.$$

For more details about the literature, we move the reader to the introduction in \cite{BB}, where the problem has been solved for ${\mathcal U}$ beeing an arbitrary collection of squares in the smooth case in \cite{BB} (i.e. with a smooth cutoff as symbol instead of the non-smooth characteristic function) and then later in the non-smooth case in \cite{BV1}. The result as also been extended to the setting of an arbitrary collection of rectangles, see \cite{BV2}. 

This current work is dedicated to the situation where ${\mathcal U}$ is a disjoint collection of parallel strips. Indeed, the geometry of a strip is motivated by the geometry of the $\BHT$ (bilinear Hilbert transform): we identify a strip $S$ (parallel to the first diagonal) to an interval $\omega_{S}$ by
$$ S:=\{(\xi,\eta)\in{\mathbb R}^2,\ \xi-\eta \in \omega_{S}\}.$$
One can then consider its (non-smooth) bilinear projection
$$ (f,g) \rightarrow \iint e^{ix(\xi+\eta)} \hat{f}(\xi) \hat{g}(\eta) {\bf 1}_{S}(\xi,\eta)d\xi d\eta,$$
which is equal to a linear combination between the identity and some modulated $\BHT$'s,
where the bilinear Hilbert transform $\BHT$ is given by
$$ \BHT(f,g)(x) = \iint e^{ix(\xi+\eta)} \hat{f}(\xi) \hat{g}(\eta) {\bf 1}_{\xi-\eta\geq 0} d\xi d\eta.$$

We recall that the $\BHT$ is bounded (see \cite{LT1,LT2}) from $L^{p_1} \times L^{p_2}$ to $L^{p}$ for every exponents $(p_1,p_2,p)$ such that their inverse belong to the range ($p'$ will always be the dual exponent of $p$, defined by $p'=\frac{p}{p-1}$ and which can be negative since $p$ can be smaller than $1$)
$${\textrm Range}_{\BHT}:=\Big\{(\frac{1}{p_1},\frac{1}{p_2},\frac{1}{p'})\in [0,1)^2 \times (-1,1),\ 0<1-\frac{1}{p'}=\frac{1}{p_1}+\frac{1}{p_2}<\frac{3}{2}\Big\}.$$
Following the analogy of Rubio de Francia's result, it is then natural to try to understand if there is also some orthogonality aspects in this bilinear setting: to an arbitrary disjoint collection of intervals, identify the collection of disjoint strips ${\mathcal S}:=(S)_{S\in{\mathcal S}}$ (such that $(\omega_S)_{S\in{\mathcal S}}$ are disjoint) and to look for boundeness of the quadratic bilinear functional 
$$ \Pi_{\mathcal S}(f,g) := \Big( \sum_{S\in{\mathcal S}} |\pi_S(f,g)|^2 \Big)^{1/2}.$$
This question has been studied by the author in \cite{B} under the condition that the strips have all the same length, which allowed to use a $\BHT$ time-frequency analysis, combined with some $\ell^2$-vector valued arguments.

Up to now, there is no results in the literature dealing with an arbitrary collection of strips and such a question seems to be quite challenging. We also move the reader to the introduction in \cite{BBLV,BBLV2} and mostly \cite[Paragraph III in Section 1.2.3]{BBLV2}, where different problems (even more difficult) related to this question have been described.

\bigskip

This work aims to provide a first answer, dealing with the smooth situation. So we start with ${\mathcal S}:=(S)_{S\in {\mathcal S}}$ a collection of {\it disjoint} strips $S$, that we identify by its interval $\omega_{S}$ with
$$ S:=\{(\xi,\eta)\in{\mathbb R}^2,\ \xi-\eta \in \omega_{S}\}.$$

Associated to $S$, we build the smooth bilinear projector
$$ \pi_{S}(f,g):= \iint e^{ix(\xi+\eta)} \hat{f}(\xi) \hat{g}(\eta) \chi_{\omega_S}(\xi-\eta) d\xi d\eta$$
with $\chi_{\omega_S}$ a compactly supported function in $\omega_{S}$ and adapted to it (i.e. for sufficiently enough integers $\alpha$, $\|\omega_{S}^{(\alpha)}\|_\infty\lesssim |\omega_S|^{-\alpha}$).
To such a collection ${\mathcal S}$, one can consider the bilinear square function
$$ \Pi_{{\mathcal S}}(f,g):= \Big( \sum_{S\in{\mathcal S}} |\pi_S(f,g)|^2 \Big)^{1/2}.$$

If in the linear setting the $L^2$-result is trivial (by $L^2$-orthogonality), there is no similar easy results in the bilinear setting. The linear $L^2$-range (where the exponent and its dual is larger than $2$) has a bilinear counterpart which is the (strict) local$-L^2$ range:
$$ {\textrm Range}_{\textrm{local}-L^2}:=\Big\{(\frac{1}{p_1},\frac{1}{p_2},\frac{1}{p'})\in (0,1)^3,\ 0<\frac{1}{p'},\frac{1}{p_1},\frac{1}{p_2}<\frac{1}{2} \quad \textrm{and} \quad 1=\frac{1}{p'}+\frac{1}{p_1}+\frac{1}{p_2}\Big\}.$$

Our main result is the following:

\begin{theorem} \label{thm:main} Let ${\mathcal S}$ be a disjoint collection of strips: $(\omega_S)_{S\in{\mathcal S}}$ are pairewise disjoint. Then for every $\epsilon>0$, there exists a $\epsilon^2$-neighborhood $V_{\epsilon}$ of the local$-L^2$ range such that for every exponents $(\frac{1}{p_1},\frac{1}{p_2},\frac{1}{p'})\in V_{\epsilon}$ satisfying $1=\frac{1}{p'}+\frac{1}{p_1}+\frac{1}{p_2}$, the square function $\Pi_{\mathcal S}$ is bounded from $L^{p_1}({\mathbb R}) \times L^{p_2}({\mathbb R})$ to $L^{p}({\mathbb R})$ with $p:=\frac{p'}{p'-1}$ and a uniform bound (with respecto to ${\mathcal S}$) of order
$$ \|\Pi_{\mathcal S}\|_{L^{p_1} \times L^{p_2} \to L^{p}} \lesssim_{\epsilon} (\sharp \mathcal S)^{\epsilon}.$$
\end{theorem}

\begin{remark} So our result is the first result dealing with an arbitrary collection of disjoint strips (up to an $\epsilon$-loss) and the statement can be thought as the bilinear version of the $L^2$-boundedness in the linear context. As we will see, the proof is far more delicate and subtle than in the linear setting !
\end{remark}

As it will be proved, $V_{\epsilon}$ can be taken for example of the form
$$ V_{\epsilon}:=\{(\frac{1}{p_1},\frac{1}{p_2},\frac{1}{p'})\in(0,1)^2 \times (-1,1),\ \max(\frac{1}{p_1},\frac{1}{p_2},\frac{1}{p'}) < \frac{1}{2}+\frac{8}{9}\epsilon^2\}.$$
We already know that
$$ \Pi_{\mathcal S}(f,g) \lesssim (\sharp {\mathcal S})^{1 \over 2} \cdot M(f,g)$$
where $M$ is the bilinear maximal function (since for every $S\in {\mathcal S}$, the smooth bilinear operator $\pi_{S}$ is pointwisely controled by $M$).
We know (see \cite{Lacey}) that $M$ is bounded in the $\BHT$ range ${\textrm Range}_{\BHT}$. So by interpolation we get the following:

\begin{theorem} \label{thm:main2} Fix exponents with $(\frac{1}{p_1},\frac{1}{p_2},\frac{1}{p'})\in {\textrm Range}_{\BHT}$. One has for every $\epsilon>0$,
$$ \| \Pi_{\mathcal S} \|_{L^{p_1} \times L^{p_2} \to L^p} \lesssim (\sharp {\mathcal S})^{\sigma(p_1,p_2,p) + \epsilon},$$
with
$$ \sigma(p_1,p_2,p):=\max\Big(\frac{1}{p_1}-\frac{1}{2},0\Big) + \max\Big(\frac{1}{p_2}-\frac{1}{2},0\Big) + \max\Big(\frac{1}{p'}-\frac{1}{2},0\Big).$$
\end{theorem}

About the optimality of the range (up to the $\epsilon$ parameter as small as we want), one coud follow an analogy between the linear theory and the bilinear theory, and  by observing that the quadratic functionals are not symmetric, we could expect the following:

\begin{conjecture} For every collection of disjoint strips ${\mathcal S}$, every $\epsilon>0$ and every exponents such that $p_1,p_2\in [2,\infty)$ then uniformly in the collection ${\mathcal S}$, we might have
$$ \| \Pi_{\mathcal S} \|_{L^{p_1} \times L^{p_2} \to L^p} \lesssim (\sharp {\mathcal S})^{\epsilon},$$
for $p$ given by $\frac{1}{p}=\frac{1}{p_1}+\frac{1}{p_2}$ (and maybe even without the loss $\epsilon$).
\end{conjecture}

\begin{remark}
\begin{itemize}
\item[(a)] A similar conjecture could also be formulated in the context of \cite{BB}, where an arbitrary collection of squares is considered; however, even this simpler case remains unsolved. While the proof does not rely on symmetric arguments for the functions $f,g$ and for the dualizing functions $\hh$, it appears that the proof only works in the stated range which is symmetric. 

\item[(b)] In the context where the strips have same (or equivalent) width, then in \cite{GPS} it has been proved the estimate for $p=2$, without any loss $\epsilon$.

\item[(c)] Moreover all these questions can be extended to the non-smooth context, where the symbol $\chi_{\omega_S}$ in $ \pi_S$ is replaced by the characteristic function ${\bf 1}_{\omega_S}$. In such a non-smooth setting, it is only known the boundedness in the strict local-$L^2$ range for strips of the same width -- see \cite{B} and an estimate is proved without any loss.
\end{itemize}
\end{remark}

\bigskip

The proof of Theorem \ref{thm:main} will follow the approach developped in \cite{BB} for squares, and that we adapt here for strips. It goes through a {\it size/energy} argument after having defined the suitable quantities and a reshuffling of the frequency geometry into {\it columns} and {\it rows}. With respect to \cite{BB}, the main difference is the following one. In \cite{BB} we were considering $\ell^{r}$-functionals over a collection of squares, for $r>2$ -- which by duality means to consider some $\ell^{r'}$-sequence of functions. If a strip can be discretized by a collection of squares, the main obstacle will be that we cannot sum over these square in $\ell^{r'}$, since $r'<2$. So the main idea will be to define a new {\it energy} for the dual functions which will be a $L^2(\ell^2)$ quantity. Indeed, since we can only hope to sum in $\ell^2$ over the squares along a fixed strip, then by 'homogenity' we have to sum in $\ell^2$ the strips as well. That will be done up to a loss in terms of $\sharp {\mathcal S}$ to the power $\frac{3}{2} \alpha$ with $\alpha:=\frac{1}{2}-\frac{1}{r}$ and it will be necessary in several occasions that $\alpha>0$ (which will then become the $\epsilon:=\frac{3}{2}\alpha>0$ in the statement).

\section{Discrete model operators and interpolation}

\subsection{Reduction to a well-distributed collection of strips}

In the linear setting, namely Rubio de Francia's work \cite{RF}, a first step is to reduce the study to the case of a {\it well-separated} collection of intervals, which is a slightly stronger condition than the disjointness: a collection of interval $(\omega)_{\omega\in \Omega}$ is said to be {\it well-separated} if\footnote{The numerical constant $2$ is not important and could be replaced by any constant $k>1$.}
$$ \textrm{the collection $(2\omega)_{\omega\in\Omega}$ is pairwise disjoint.}$$

In this current work, we will need to discretize the strips by squares (in order to apply some time-frequency analysis involving {\it tiles}) which will have to be mutually disjoint (when coming from different strips). Hence we will need to use the property of {\it well-separated} and we first explain how in the smooth case, this is easily doable to assume.

\begin{proposition} To prove Theorem \ref{thm:main}, we may assume that the collection of strips or more precisely of intervals $(\omega_S)_{S\in{\mathcal S}}$ is {\it well-distributed}.
\end{proposition}

\begin{proof} We start with an initial collection $(\omega_S)_{S\in{\mathcal S}}$ which is only disjoint. Then for every interval $\omega_S$ we consider its Whitney covering (partition) by dyadic intervals: $\omega_S:= \bigsqcup \omega$ with dyadic intervals $\omega \subset \omega_S$ such that
$$ 4 |I| \leq \textrm{dist}(\omega, \omega_S^c) \leq 10 |I|$$
and denote for $k\geq 0$
$$ \Omega_S^k:=\{\omega, 2^{-k-1} |\omega_S| < |\omega| \leq 2^{-k} |\omega_S|\}$$ so that we have 
$$ \omega_S = \bigsqcup_{k\geq 0} \big( \bigsqcup_{\omega\in\Omega_S^k} \omega\big).$$
We then consider a partition of the unity $(\chi_{\omega})_{\omega}$ associated to the covering $(2\omega)_{\omega}$, so that
$$ 1_{\omega_S} = \sum_{\omega} \chi_{\omega}$$
and each $\chi_{\omega}$ beeing adapted to $\omega$, which means that $\chi_{\omega}\in C^\infty_0(2\omega)$ and $\|\chi_{\omega}^{(n)}\|_\infty \lesssim |\omega|^{-n}$ for sufficiently large integers $n$.

So then we decompose
$$ \chi_{\omega_S} = \sum_{k\geq 0} \chi_{\omega_S} \cdot \big(\sum_{\omega \in \Omega_S^k} \chi_{\omega}\big) := \sum_{k\geq 0} \chi_{\omega_S,k}.$$ 
By the properties of construction, $\chi_{\omega_{S,k}}$ is supported and adapted to the union of two intervals which constitute a {\it well-separated} collection and are uniformly bounded by $2^{-k}$. So if we assume that the result is true for such collections, we deduce that for every $k\geq 0$, the corresponding square function $\Pi_{{\mathcal S},k}$ is bounded with an extra $2^{-k}$ and then we conclude by using that pointwisely
$$ \Pi_{\mathcal S}(f,g) \leq \sum_{k\geq 0} \Pi_{{\mathcal S},k}(f,g).$$ 
\end{proof}

So from now on and in all the remaining part of this work, we will assume that the collection of intervals $(\omega_S)_{S\in{\mathcal S}}$ is {\it well-separated}.

The next step of the reduction is to discretize the square function and to use multilinear interpolation to reduce to {\it restricted weak-type estimates} for the discrete models.

\subsection{Reduction to discrete model operators}

By duality, to estimate $\Pi_{\mathcal S}$ in $L^p$ ($1<p<\infty$), one can test it and use
$$ \|\Pi_{\mathcal S}(f,g)\|_{L^p} := \sup_{\hh\in L^{p'}(\ell^2)  \atop \|\hh\|_{L^{p'}(\ell^2)}=1} \int_{\mathbb R} \sum_{S\in{\mathcal S}} \pi_{S}(f,g)(x) h_{S}(x)\, dx$$
where we take the supremum over the sequence of functions $\hh:=(h_S)_{S\in {\mathcal S}}$ belonging to $L^{p'}(\ell^2({\mathcal S}))$ and beeing normalized.

In this part $\pi_{S}$ is the smooth bilinear projector, so by covering the strip $S$ with squares $\omega$ (of scale the one of the strip) and then by performing a windowed Fourier decomposition, it is rather standard that $\pi_S(f,g)$ can be decomposed along wave-packets:

\begin{definition}[Wave-packet] For a rectangle $I \times U$ of area $1$, a wave-packet $\phi_{I\times U}$ is a smooth function $L^2$-normalized which has a frequency support on $U$ and is adapted to $U$ in frequence and to $I$ in space, i.e. for sufficiently large integers $n\geq 0$ and $M\geq 0$, for every $x,\xi$ then
$$ |\phi_{I\times U}^{(n)}(x)|\lesssim |I|^{-n-\frac{1}{2}} \big(1+\frac{d(x,I)}{|I|}\big)^{-M} \qquad \textrm{and} \qquad |\hat \phi_{I\times U}^{(n)}(\xi)|\lesssim |\omega|^{-n-\frac{1}{2}} {\bf 1}_{\omega}(\xi).$$
For an interval $I$, we denote
$$ \chi(x):=\big(1+\frac{d(x,I)}{|I|}\big)^{-M}$$
where $M$ is an exponent as large as required (and which can vary from a line to another one).

\end{definition}

By standard arguments, we then have the following decomposition \footnote{Here we denote $\approx$ by the fact that the LHS can be decomposed into a (finite or infinite with fastly decreasing coefficients) sum of terms of the form the RHS.}:
$$ \pi_S(f,g)= \sum_{\omega \cap S \neq \emptyset \atop \omega=\omega_1 \times \omega_2} 
 \sum_{I \atop {|I|=|\omega_1|^{-1} \atop |I|=|\omega_2|^{-1}}} |I|^{-1/2} \cdot \langle f, \phi_{I\times \omega_1}\rangle \cdot \langle g,\phi_{I \times \omega_2}\rangle \cdot \phi_{s_3},$$ 
where the sum can be restricted to intervals $\omega$ and $I$ belonging to a dyadic grid (up to a finite sum of similar terms in order to take into account shifted dyadic grids as well).

So it is sufficient to work on such discrete models. Aiming that we fix ${\mathbb D}$ the collection of dyadic intervals and we recall the notion of {\it tri-tiles}:

\begin{definition} A (unitary) tile is the product of two dyadic intervals $I\times v$ of area one $|I| \cdot |v|=1$. A tri-tile $s=(s_1,s_2,s_3)$ is a collection of three tiles sharing the same spatial interval $I=I_s$
$$ s_1=I_s \times \omega_{s_1}, \qquad s_2=I_s \times \omega_{s_2} \qquad \textrm{and} \qquad s_2=I_s \times \omega_{s_3}$$
with the property that $\omega_{s_3} \subset -(\omega_{s_1}+\omega_{s_2})$. We denote by $\PP$ for a generic (finite) collection of {\it tri-tiles}.
For such a {\it tri-tile} $s$, we will denote the frequency square $\omega_{s}:=\omega_{s_1} \times \omega_{s_2}$ which will play an important role, interacting with the strips.
\end{definition}

With these notions, we are then reduced to the study of the following trilinear form (and to prove bounds, uniformly with respect to any finite collection $\PP$)
$$ \Lambda_{\PP}(f,g,\hh ) := \sum_{S \in {\mathcal S}} \sum_{s\in \PP \atop \omega_s \cap S \neq \emptyset} |I_s|^{-1/2} \cdot |\langle f, \phi_{s_1}\rangle| \cdot |\langle g,\phi_{s_2}\rangle| \cdot |\langle h_S, \phi_{s_3}\rangle|,$$  
where $\hh:=(h_S)_{S\in{\mathcal S}}$ is a sequence of functions (indexed by the collection of strips) and by $\omega_s \cap S \neq \emptyset$, we mean that $\omega_s:=\omega_{s_1} \times \omega_{s_2}$ has to meet $S$ and so in particular $(\omega_{s_i})_{i=1,2,3}$ have to be at the scale of $S$ (and more precisely of $\omega_S$). When $\omega_s \cap S\neq \emptyset$ then $\omega_s \subset 2S$ where $2S$ is the strip given by the double interval $\omega_{2S}:=2\omega_{S}$. Because we assume that the collection of strips ${\mathcal S}$ is {\it well-distributed} that means that such a frequency square will not meet any other strips. So for every $s\in\PP$ there is at most one strip such that $\omega_s\cap S\neq \emptyset$. That is an important property and will be implicitely used at many places.

\medskip

As usual, we will prove boundedness by use of multilinear interpolation and restricted weak-type estimates. For $E$ any measureable subset of finite measure, we denote by
$$ X(E):=\{f\in L^{\infty}({\mathbb R}),\ |f(x)|\leq {\bf 1}_{E}(x) \ a.e.\}$$
and its $\ell^2$ version (we keep the same notation for simplicity, the context will dictate the use)
$$ X(E):=\{\hh\in L^{\infty}(\ell^2({\mathbb R})),\ \|\hh(x)\|_{\ell^2}\leq {\bf 1}_{E}(x)\ a.e.\}.$$

In order to prove Theorem \ref{thm:main}, we know that by multilinear interpolation, it is sufficient to prove some ``restricted weak-type'' estimates.

\begin{definition} For a triple $\nu:=(\nu_1,\nu_2,\nu_3)$ satisfying $\nu_1+\nu_2+\nu_3=1$ with $0<\nu_1,\nu_2<1$ and $-1<\nu_3<1$, a trilinear form ${\mathcal L}$ acting on the Schwartz spaces ${\mathcal S}({\mathbb R}) \times {\mathcal S}({\mathbb R}) \times \ell^2({\mathcal S}({\mathbb R}))$ is said to be of restricted weak-type $\nu$ if for every measurable sets (of finite measure) $F,G,H$, there exists a major subset $H'\subset H$ (with $|H|\leq 2|H'|$) such that for every functions $f\in X(F)$, $g\in X(G)$ and $\hh \in X(H')$ one has
$$ \left|{\mathcal L}(f,g,\hh)\right| \lesssim |F|^{\nu_1} \cdot |G|^{\nu_2} \cdot |H|^{\nu_3}$$
and we denote by $\|{\mathcal L}\|_{[\nu_1,\nu_2,\nu_3]}$ the best implicit constant.
\end{definition}

The multilinear interpolation (with a slight precaution due to the vector-valued context, see \cite[Section 2.2]{TheseCristina} and \cite[Section 2.3]{BeneaMuscalu}) gives the following:

\begin{theorem} Let $p_1,p_2\in(1,\infty)$ and $p\in(\frac{1}{2},\infty)$ such that $\frac{1}{p_1}+\frac{1}{p_2}=\frac{1}{p}$.
Then with $\frac{1}{p'}=1-\frac{1}{p}\in(-1,1)$ if there exists a neighborhood $V$ around $(\frac{1}{p_1},\frac{1}{p_2},\frac{1}{p'})$ such that for every $\nu\in V$, the trilinear (or sublinear) form ${\mathcal L}$ is of restricted weak-type $\nu$, then ${\mathcal L}$ is bounded (admits a countinous  extension) from $L^{p_1}\times L^{p_2} \times L^{p'}(\ell^2)$ (if $p'\in[1,\infty)$) and
$$ \|{\mathcal L}\|_{L^{p_1} \times L^{p_2} \times L^{p'}(\ell^2)} \lesssim \sup_{\nu\in V} \|{\mathcal L}\|_{[\nu_1,\nu_2,\nu_3]}.$$
In the case that ${\mathcal L}(f_1,f_2,\hh)=\langle T(f,g),\hh\rangle$ for some $\ell^2$-valued bilinear operator $T$ then we have (even if $p:=\frac{p'}{p'-1}<1$) that $T$ admits a continuous extension from $L^{p_1} \times L^{p_2}$ into $L^{p}(\ell^2)$ with
$$ \|T\|_{L^{p_1} \times L^{p_2} \to L^{p}(\ell^2)} \lesssim \sup_{\nu\in V} \|{\mathcal L}\|_{[\nu_1,\nu_2,\nu_3]}.$$
\end{theorem}

From all these previous reductions, we get that Theorem \ref{thm:main} will be a consequence of this one:

\begin{theorem} \label{thm:main:reduc} Fix $\epsilon>0$ arbitrarily small and consider ${\mathcal S}$ a collection of {\it well-distributed} strips. Then, there exists a $\epsilon^2$-neighborhood $W_{\epsilon}$ of the local$-L^2$ range such that for every $\nu=(\nu_1,\nu_2,\nu_3)\in W_{\epsilon}$ satisfying $1=\nu_1+\nu_2+\nu_3$ and for all finite collection of tri-tiles $\PP$ (and uniformly with respect to it) then the trilinear form $\Lambda_{\PP}$ is of restricted weak-type $\nu$ with the bound
$$ \|\Lambda_{\PP}\|_{[\nu_1,\nu_2,\nu_3]} \lesssim (\sharp \mathcal S)^{\epsilon}.$$
\end{theorem}

\section{The reshuffling of the  collection of tiles}

From now, we fix the collection of tri-tiles $\PP$ and we aim to study the trilinear form $\Lambda_{\PP}$. Following \cite{BB}, we re-shuffle the collection in terms of {\it rows} and {\it columns}, with the following definitions.

\begin{definition}[Row/Column] A sub-collection $\RR \subset \PP$ is a row of top $t\in \PP$ if for every $s\in \RR$, 
\[
I_s \subseteq I_t \qquad \text{and} \qquad \omega_{t_2} \subseteq \omega_{s_2}.
\]
We denote the top as $t_{\RR}:=I_{\RR} \times \omega_{\RR}$.

A sub-collection $\CC \subset \PP$ is a column of top $t\in\PP$ if for every $s\in \CC$, 
\[
I_s \subseteq I_t \qquad \text{and} \qquad \omega_{t_1} \subseteq \omega_{s_1}.
\]
We denote the top as $t_{\CC}:=I_{\CC} \times \omega_{\CC}$.
\end{definition}

\begin{remark} Due to the geometry in the frequency plane, we observe that the tiles $s_1:=I_{s} \times \omega_{s_1}$ are mutually disjoint when $s$ varies along a row.
Similarly, the tiles $s_2:=I_{s} \times \omega_{s_2}$ are mutually disjoint when $s$ varies along a column.
\end{remark}

As expected (see \cite[Definition 2.4]{BB}), we say that a sequence of columns $(\CC_1,\cdots , \CC_n)$ is mutually disjoint if they are are disjoint sets of tri-tiles and if $(I_{\CC_j} \times \omega_{\CC_{j,1}})_{j}$ are disjoint as well.
A sequence of rows $(\RR_1,\cdots, \RR_n)$ is mutually disjoint if they are are disjoint sets of tri-tiles and if $(I_{\RR_j} \times \omega_{\CC_{j,2}})_{j}$ are disjoint as well.

\begin{definition}[Tree/Forest]
A {\it tree} $\TT$ is a collection which is either a row or a column and we denote by $I_{\TT}$ the spatial interval of its top. 
A {\it forest} $\FF$ is an arbitrary finite collection of {\it trees}.
\end{definition}

Then, to perform the time-frequency analysis, we need to define suitable quantities (see \cite[Definition 3.1]{BB}).

\begin{definition} For a sub-collection $\PP' \subset \PP$, we define the {\it size} for $f$ by
$$ \size_{\PP'}(f) := \sup_{s \subset \PP'} \frac{|\langle f, \phi_{s_1} \rangle |}{|I_s|^{1\over 2}}.$$
For a sub-collection $\PP' \subset \PP$, we define the {\it size} for $g$ by
$$ \size_{\PP'}(g) := \sup_{s \subset \PP'} \frac{|\langle g, \phi_{s_2} \rangle |}{|I_s|^{1 \over 2}}.$$
\end{definition}

For the sequence of functions $\hh$, we aim to keep a definition in terms of wave-packet coefficients. We refer the reader to the proof of \cite[Proposition 2.5]{BB} which motivated the notion of {\it size} through a maximal function (\cite[Definition 3.2]{BB}). Here we will keep the quantity appearing at the beginning of the proof and involving the wave-packet coefficients: through all the rest of the work we will use an extra exponent $r>2$ or a parameter $\alpha>0$ jointly defined by $\alpha:=\frac{1}{2}-\frac{1}{r}$ which will be used  in the technical estimates and will then be related to the $\epsilon$ parameter in the main statement.

\begin{definition}
For a sub-collection $\PP' \subset \PP$, we define the {\it size} for $\hh$, as 
$$ \size_{\PP'}(\hh) := \sup_{\TT \subset \PP' \atop \textrm{ column or row}} \left(\frac{1}{|I_{\TT}|^{r'/2}} \sum_{S\in \mathcal{S}} \sum_{I \atop |I|=|\omega_S|^{-1}} \sum_{s\in \TT \atop \omega_s \cap S\neq \emptyset} \big(\frac{|I_s|}{|I_{\TT}|}\big)^{\alpha r'} |\langle h_S,\phi_{s_3}\rangle|^{r'}\right)^{1/r'}.$$ 
\end{definition}

\begin{remark}
Since $\frac{r'}{2}+\alpha r' = 1$, one has
$$ \size_{\PP'}(\hh) := \sup_{\TT \subset \PP' \atop \textrm{ column or row}} \left(\frac{1}{|I_{\TT}|} \sum_{S\in \mathcal{S}} \sum_{s\in \TT \atop \omega_s \cap S\neq \emptyset} |I_s|^{\alpha r'} |\langle h_S,\phi_{s_3}\rangle|^{r'}\right)^{1/r'}.$$ 
\end{remark}

We will also use the ``modified'' {\it sizes}, encoding only the spatial information: for a function $v$
$$ \ssize_{\PP'}(v):=\sup_{s\in \PP'} \Big(\frac{1}{|I_s|} \int |v(x)| \chi_{I_s}(x) dx \Big).$$

We now define the {\it energy} quantities. For $f,g$ we follow \cite[Definition 3.3]{BB}:

\begin{definition}
For a sub-collection $\PP' \subset \PP$, we define the {\it energy} for $f$, as 
$$\energy_{\PP'}(f):=\sup_{n\in{\mathbb Z} \atop \mathfrak{C}} 2^{n} \left(\sum_{\CC \in \mathfrak{C}} |I_{\CC}| \right)^{1/2}$$
where $\mathfrak{C}$ ranges over all collections of mutually disjoint columns $\CC\subset \PP'$ so that
$$ \frac{|\langle f, \phi_{s_1}\rangle |}{|I_s|^{1\over 2}} \leq 2^{n+1} \qquad \textrm{for all $s\in\CC$}$$
and whose tops $t_{\CC}$ satisfy
$$ \frac{|\langle f, \phi_{t_{\CC,1}}\rangle |}{|I_{\CC}|^{1\over 2}}\geq 2^{n} \qquad \textrm{for all $\CC\in\mathfrak{C}$}.$$

Similarly for the function $g$, with rows instead of columns.
\end{definition}

For the sequence of functions $\hh$, we will have to modify slightly the {\it energy} with respect to \cite{BB} and so we set:

\begin{definition}
For a sub-collection $\PP' \subset \PP$, we define the {\it energy} for $\hh$, as 
$$ \energy_{\PP'}(\hh):=\sup_{n\in{\mathbb Z} \atop \FF} 2^{n} \left(\sum_{\TT \in \FF} |I_{\TT}| \right)^{1/2}$$
where  $\FF$ ranges over all forests of mutually disjoint rows and mutually disjoint columns satisfying for every $\TT\in \FF$
$$ \frac{1}{|I_{\TT}|^{r'/2}} \sum_{S\in \mathcal{S}} \sum_{s\in \TT \atop \omega_s \cap S\neq \emptyset} \big(\frac{|I_s|}{|I_{\TT}|}\big)^{\alpha r'} |\langle h_S,\phi_{s_3}\rangle|^{r'} \geq 2^{nr'}.$$
\end{definition}

As we will see later (Proposition \ref{prop:energy2}), this definition allows to have a $L^2$-control of the {\it energy} which was not the case in \cite{BB}. However this is only possible here because we allow a loss in terms of $\sharp{\mathcal S}$ (as we will have in Proposition \ref{prop:energy2}). In \cite{BB}, the setting was simpler (a collection of squares instead of strips as here) but the result was stronger since there was no loss and for that it was necessary to consider the $L^{r'}$-{\it energy} introduced in \cite{BB}.

\medskip

To control the trilinear form $\Lambda_{\PP}$, we will re-suffle the whole collection $\PP$ into rows and columns (as in \cite{BB}) and so we have first to estimate the trilinear form on these elementary sub-collections. That was already done in \cite[Proposition 2.5]{BB} (so we do not repeat its proof which is relatively easy):

\begin{proposition}[Column/Row estimate] \label{prop:tree} Let $\PP'$ be a sub-collection of $\PP$ and $\CC$ be a column of $\PP'$. Then we have the following estimate:
\begin{align*}
\Lambda_{\CC}(f,g,\hh) & \lesssim \size_{\PP'}(f) \cdot (\size_{\PP'}(g))^{2\alpha} \cdot \left(\frac{1}{|I_{\CC}|} \sum_{s\in \CC} |\langle g, \phi_{s_2}\rangle |^2\right)^{\frac{1-2\alpha}{2}} \cdot \size_{\PP'}(\hh) \cdot |I_{\CC}|.
\end{align*}
If $\RR$ is a row of $\PP'$, then we have similarly
\begin{align*}
\Lambda_{\RR}(f,g,\hh) & \lesssim (\size_{\PP'}(f))^{\alpha} \cdot \left(\frac{1}{|I_{\RR}|} \sum_{s\in \RR} |\langle f, \phi_{s_1}\rangle |^2\right)^{\frac{1-2\alpha}{2}} \cdot \size_{\PP'}(g) \cdot \size_{\PP'}(\hh) \cdot |I_{\RR}|.
\end{align*}
\end{proposition}

\section{Control of the {\it sizes} and {\it energies}}

As usual the {\it size} quantities are bounded by the $L^\infty$ norm and more precisely local $L^1$-averages (or $L^{r'}$-averages for $\hh$):

\begin{proposition}[{\it Size} estimates] \label{prop:size} We have for an arbitrary collection $\PP'\subset \PP$ the following estimates:
$$ \size_{\PP'}(f) \lesssim \sup_{s\in\PP'} \left( \frac{1}{|I_s|} \int |f(x)| \chi_{I_s}(x) dx \right) \lesssim \|f\|_\infty,\quad \size_{\PP'}(g) \lesssim \sup_{s\in\PP'} \left( \frac{1}{|I_s|} \int |g(x)| \chi_{I_s}(x) dx \right) \lesssim\|g\|_\infty$$
and with $h:=\big(\sum_{S\in{\mathcal S}} |h_{S}|^2\big)^{1/2}$
$$ \size_{\PP'}(\hh) \lesssim (\sharp{\mathcal S})^{\alpha} \cdot \sup_{\TT} \left(\frac{1}{|I_{\TT}|}\int |h(x)|^{r'} \chi_{I_{\TT}}(x) dx \right)^{1/r'} \lesssim 
(\sharp{\mathcal S})^{\alpha} \cdot \|h \|_\infty.$$
\end{proposition}

\begin{proof} For $f$ and $g$, there is nothing to prove since the wave-packets are adapted to their spatial interval.

For $\hh$, we follow  what was done for the proof of \cite[Proposition 2.5]{BB}.
So we fix a {\it tree} $\TT$ and we have (with $M$ beeing the Hardy-Littlewood maximal function)
\begin{align*}
\frac{1}{|I_{\TT}|} \sum_{S\in{\mathcal S}} \sum_{s\in \TT \atop \omega_s\cap S \neq \emptyset} |I_s|^{\alpha r'} |\langle h_S,\phi_{s_3}\rangle|^{r'} & \leq  \frac{1}{|I_{\TT}|} \sum_{S\in{\mathcal S}} \sum_{s\in \TT \atop \omega_s\cap S \neq \emptyset} |I_s|^{\alpha r' + \frac{r'}{2}} \Big[\inf_{I_s}  M(h_S \chi_{I_{\TT}})\Big]^{r'} \\
& \leq  \frac{1}{|I_{\TT}|} \sum_{S\in{\mathcal S}} \sum_{s\in \TT \atop \omega_s\cap S \neq \emptyset} |I_s|^{\alpha r' + \frac{r'}{2}-1}  \cdot \int_{I_s} \big[M(h_S \chi_{I_{\TT}})(x)\big]^{r'} dx.
\end{align*}
Since $r'(\alpha+\frac{1}{2})=1$, one has 
\begin{align*}
\frac{1}{|I_{\TT}|} \sum_{S\in{\mathcal S}} \sum_{s\in \TT \atop \omega_s\cap S \neq \emptyset} |I_s|^{\alpha r'} |\langle h_S,\phi_{s_3}\rangle|^{r'} &  \leq  \frac{1}{|I_{\TT}|} \sum_{S\in{\mathcal S}} \sum_{s\in \TT \atop \omega_s\cap S \neq \emptyset} \int_{I_s} \big[M(h_S \chi_{I_{\TT}})(x)\big]^{r'} dx.
\end{align*}
We note that for $S\in{\mathcal S}$ fixed, then along any tree $\TT$ there is only one frequency square $(\omega_s)_{s\in \TT}$ which meets $S$ and so the corresponding spatial intervals $I_s$ are disjoint (and included in $I_{\TT}$). So
\begin{align*}
\frac{1}{|I_{\TT}|} \sum_{S\in{\mathcal S}} \sum_{s\in \TT \atop \omega_s\cap S \neq \emptyset} |I_s|^{\alpha r'} |\langle h_S,\phi_{s_3}\rangle|^{r'} &  \leq  \frac{1}{|I_{\TT}|} \sum_{S\in{\mathcal S}} \int_{I_{\TT}} \big[M(h_S \chi_{I_{\TT}})(x)\big]^{r'} dx \\
& \lesssim  \frac{1}{|I_{\TT}|} \sum_{S\in{\mathcal S}} \int |h_S(x)|^{r'} \cdot \chi_{I_{\TT}}(x) dx,
\end{align*}
where we used the $L^{r'}$-boundedness of the maximal function. Finally by doing a H\"older inequality along ${\mathcal S}$, one deduces 
\begin{align*}
\frac{1}{|I_{\TT}|} \sum_{S\in{\mathcal S}} \sum_{s\in \TT \atop \omega_s\cap S \neq \emptyset} |I_s|^{\alpha r'} |\langle h_S,\phi_{s_3}\rangle|^{r'} & \lesssim (\sharp {\mathcal S})^{\alpha r'} \cdot \frac{1}{|I_{\TT}|}  \int |h(x)|^{r'} \cdot \chi_{I_{\TT}}(x) dx.
\end{align*}
Taking the supremum over all {\it trees} in $\PP'$, allows us to conclude the estimate.
\end{proof}

For the functions $f,g$, we have the standard energy estimate -- see \cite[Propositions 3.6 and 5.1]{BB}:

\begin{proposition}[{\it Energy} estimates] \label{prop:energy1} For every sub-collection $\PP'\subset \PP$ of tri-tiles, one has
$$ \energy_{\PP'}(f)\lesssim \|f\|_2 \qquad \textrm{and} \qquad \energy_{\PP'}(g) \lesssim \|g\|_2.$$
If the collection $\PP'$ is localized on a spatial interval $I_{\PP'}$ then we have the local estimates
$$ \energy_{\PP'}(f)\lesssim \|f \cdot \chi_{I_{\PP'}}\|_2 \qquad \textrm{and} \qquad \energy_{\PP'}(g) \lesssim \|g \cdot \chi_{I_{\PP'}}\|_2.$$
\end{proposition}

We aim to have a similar result for the sequence of functions $\hh$.

\begin{proposition}[{\it Energy} estimates] \label{prop:energy2} Let $\PP'$ be any sub-collection of tri-tiles of $\PP$ and $\hh$ a sequence of functions. Then with $h:=(\sum_{S\in {\mathcal S}} |h_S|^2)^{1/2}$, we have
$$ \energy_{\PP'}(\hh) \lesssim  (\sharp {\mathcal S})^{\alpha} \cdot \|h\|_2.$$
If the collection $\PP'$ is localized on a spatial interval $I_{\PP'}$ then we have the local estimate
$$ \energy_{\PP'}(\hh) \lesssim  (\sharp {\mathcal S})^{\alpha} \cdot \|h \cdot \chi_{I_{\PP'}}\|_2.$$
\end{proposition}

\begin{proof}
Let us chose a maximiser in the definition of the {\it energy}: an integer $n$ and a mutually disjoint forest $\FF$ such that 
$$\energy_{\PP'}(\hh) \leq 2^{n+1}\left(\sum_{\TT \in \FF} |I_{\TT}| \right)^{1/2}$$
and for every tree $\TT\in\FF$
$$ 2^{n r'} |I_{\TT}| \leq \sum_{S\in{\mathcal S}} \sum_{s\in\TT \atop \omega_s \cap S \neq \emptyset} |I_s|^{\alpha r'} |\langle h_S,\phi_{s_3}\rangle|^{r'}.$$
With $c_s:=|I_s|^{\alpha} |\langle h_S, \phi_{s_3}\rangle |$ (where implicitly $S$ is the unique strip in ${\mathcal S}$ intersecting $\omega_{s}$), we have
\begin{align*}
\sum_{\TT\in \FF} \sum_{s\in \TT} |c_s|^{r'} & =
\sum_{\TT\in \FF} \sum_{S\in {\mathcal S}} \sum_{s\in \TT \atop \omega_s \cap S \neq\emptyset} |I_s|^{\alpha r'} |\langle h_S,\phi_{s_3}\rangle|^{r'} \\
& \leq
\sum_{\TT\in \FF} \sum_{S\in {\mathcal S}} \big(\sum_{s\in \TT \atop \omega_s \cap S \neq\emptyset} |I_s|\big)^{\alpha r'} \big(\sum_{s\in \TT \atop \omega_s \cap S \neq\emptyset} |\langle h_S,\phi_{s_3}\rangle|^{2}\big)^{r'/2} \\
& \leq
\sum_{\TT\in \FF} \sum_{S\in {\mathcal S}} |I_{\TT}|^{\alpha r'} \big(\sum_{s\in \TT \atop \omega_s \cap S \neq\emptyset} |\langle h_S,\phi_{s_3}\rangle|^{2}\big)^{r'/2} \\
\end{align*}
where we used H\"older inequality with $\frac{1}{r'}=\alpha+\frac{1}{2}$ and then that for $S\in{\mathcal S}$ beeing fixed, then the frequency tile $\omega_s$ is fixed when $s$ varies along a tree $\TT$ and so the spatial intervals $I_s$ are disjoint (and included in $I_{\TT}$).
So we obtain (by H\"older inequality)
\begin{align*}
\sum_{\TT\in \FF} \sum_{s\in \TT} |c_s|^{r'} & \leq (\sharp {\mathcal S})^{1-\frac{r'}{2}}
\sum_{\TT\in \FF} |I_{\TT}|^{\alpha r'} \big(\sum_{S\in {\mathcal S}} \sum_{s\in \TT \atop \omega_s \cap S \neq\emptyset} |\langle h_S,\phi_{s_3}\rangle|^{2}\big)^{r'/2} \\
& \leq (\sharp {\mathcal S})^{1-\frac{r'}{2}}
\big(\sum_{\TT\in \FF} |I_{\TT}|\big)^{\alpha r'} \big(\sum_{\TT\in \FF} \sum_{S\in {\mathcal S}} \sum_{s\in \TT \atop \omega_s \cap S \neq\emptyset} |\langle h_S,\phi_{s_3}\rangle|^{2}\big)^{r'/2}.
\end{align*}
By easy orthogonality arguments (since for $S\in{\mathcal S}$ fixed, the tri-tiles $s$ such that $\omega_s\cap S\neq \emptyset$ have all the same scale, given by the strip $S$)
\begin{align*}
\big(\sum_{\TT\in \FF} \sum_{S\in {\mathcal S}} \sum_{s\in \TT \atop \omega_s \cap S \neq\emptyset} |\langle h_S,\phi_{s_3}\rangle|^{2}\big)^{1/2} & \lesssim \big(\sum_{S\in {\mathcal S}} \|h_S \|_2^2\big)^{1/2} \\
& \lesssim \|h\|_{2},
\end{align*}
with by convention $h:=(\sum_{S\in {\mathcal S}} |h_S|^2)^{1/2}$. So we conclude to
\begin{align*}
2^{nr'} \big(\sum_{\TT\in \FF} |I_{\TT}| \big) & \leq
\sum_{\TT\in \FF} \sum_{s\in \TT} |c_s|^{r'}  \lesssim (\sharp {\mathcal S})^{1-\frac{r'}{2}}
\big(\sum_{\TT\in \FF} |I_{\TT}|\big)^{\alpha r'} \|h\|_2^{r'} 
\end{align*}
which yields (since $1-\alpha r'=\frac{r'}{2}$)
$$ 2^{n} \big(\sum_{\TT\in \FF} |I_{\TT}| \big)^{\frac{1}{2}} \lesssim (\sharp {\mathcal S})^{\alpha} \cdot \|h\|_2$$
and so allows us to conclude to the (global) energy estimate.

For the localized statement, we just observe that if $\PP'$ is localized on $I_{\PP'}$ then all the previous arguments imply wave-packets also localized in $I_{\PP'}$ and so one can keep and track the localization at every step.
\end{proof}

\section{Decomposition lemmas and summation over columns/rows}

Through all this section and the next one, we fix the functions $f,g,\hh$ and a sub-collection of tri-tiles $\PP'\subset \PP$ and we will use the following notations
\begin{gather*}
S_1:=\size_{\PP'}(f),\quad S_2:=\size_{\PP'}(g),\quad S_3:=\size_{\PP'}(\hh),\\
 E_1:=\energy_{\PP'}(f), \quad E_2:=\energy_{\PP'}(g), \quad E_3:=\energy{\PP'}(\hh)
\end{gather*}
for the {\it sizes} and {\it energies}.

For $f$ and $g$, we recall the following ``selection algorythm'' -- see \cite[Lemma 3.9]{BB}:

\begin{lemma} \label{lemma:selection:f} Let $\QQ$ be a sub-collection of $\PP'$ and an integer $n_0$ such that $\size_{\QQ}(f)\leq 2^{-n_0} E_1$. Then one can extract a collection $\mathfrak{C}$ of columns such that 
\begin{itemize}
\item the remaining collection $\QQ':=\QQ\setminus (\bigcup_{\CC\in{\mathfrak{C}}} \CC)$ satisfies $\size_{\QQ'}(f)\leq 2^{-n_0-1} E_1$
\item the selected columns satisfy 
$$\sum_{\CC\in{\mathfrak{C}}} |I_{\CC}| \lesssim 2^{2n_0}.$$
\end{itemize}
\end{lemma}

We have a similar lemma for the function $g$ with extracting a collection of rows -- see \cite[Lemma 3.10]{BB}:

\begin{lemma} \label{lemma:selection:g} Let $\QQ$ be a sub-collection of $\PP'$ and an integer $n_0$ such that $\size_{\QQ}(g)\leq 2^{-n_0} E_2$. Then one can extact a collection $\mathfrak{R}$ of rows such that 
\begin{itemize}
\item the remaining collection $\QQ':=\QQ\setminus (\bigcup_{\RR\in{\mathfrak{R}}} \RR)$ satisfies $\size_{\QQ'}(g)\leq 2^{-n_0-1} E_2$
\item the selected rows satisfy 
$$\sum_{\RR\in{\mathfrak{R}}} |I_{\RR}| \lesssim 2^{2n_0}.$$
\end{itemize}
\end{lemma}

We aim to have a similar lemma for the sequence $\hh$ -- see \cite[Lemma 3.11]{BB} (that we adapt to our new definition of {\it energy} in this current setting):

\begin{lemma} \label{lemma:selection:h}
Let $\QQ$ be a sub-collection of $\PP'$ and an integer $n_0$ such that $\size_{\QQ}(\hh)\leq 2^{-n_0} E_3$. Then one can extact a collection $\mathfrak{R}$ of rows and a collection $\mathfrak{C}$ of columns (both mutually disjoint) such that 
\begin{itemize}
\item the remaining collection $\QQ':=\QQ\setminus (\bigcup_{\RR\in{\mathfrak{R}}} \RR \cup \bigcup_{\CC\in{\mathfrak{C}}} \CC)$ satisfies $\size_{\QQ'}(\hh)\leq 2^{-n_0-1} E_3$
\item the selected rows and columns satisfy 
$$ \sum_{\RR\in{\mathfrak{R}}} |I_{\RR}| + \sum_{\CC\in{\mathfrak{C}}} |I_{\CC}| \lesssim 2^{2 n_0}.$$
\end{itemize}
\end{lemma}

By iterating these three previous ``selection lemmas'', one has the following decomposition:

\begin{proposition} \label{prop:decomposition} Let fix an arbitrary collection $\PP'\subset \PP$ of tri-tiles. Then we have a decomposition
$$  \PP'= \bigcup_{n\in{\mathbb Z}} \QQ_n$$
where $\QQ_{n}$ is given by the union of a forest of rows $\FF_{n}^{1}$ and a forest of columns $\FF_{n}^2$ and satisfy
\begin{itemize}
\item Structural decomposition: $\QQ_n = (\bigcup_{\RR \in \FF_n^1} \RR) \cup (\bigcup_{\CC \in \FF_n^2} \CC)$; 
\item Control of the size
\begin{align*}
\size_{\QQ_n}(f) \leq \min(2^{-n} E_1,S_1), \quad \size_{\QQ_n}(g) \leq \min(2^{-n} E_2,S_2), \quad \size_{\QQ_n}(\hh) \leq \min(2^{-n} E_3,S_3);
\end{align*}
\item Control of the forests
$$ \sum_{\RR\in\FF_{n}^1} |I_{\RR}|\lesssim 2^{2n} \qquad \textrm{and} \qquad \sum_{\CC\in\FF_{n}^2} |I_{\CC}|\lesssim 2^{2n}.$$
\end{itemize}
Moreover $\FF_{n}^1$ is empty if $n$ is such that
$$2^{-n} E_1 \geq 2 S_1 \qquad \textrm{and} \qquad 2^{-n} E_3 \geq 2 S_3.$$
Similarly, $\FF_{n}^2$ is empty if $n$ is such that
$$2^{-n} E_2 \geq 2 S_2 \qquad \textrm{and} \qquad 2^{-n} E_3 \geq 2 S_3.$$
\end{proposition}

\section{Boundedness of the trilinear forms}

We first prove the following generic estimate: we consider a fix sub-collection of tri-tiles $\PP'\subset \PP$, measurable subsets $F,G,H$ and functions $f\in X(F)$, $g\in X(G)$, $\hh\in X(H)$. For each of them, as in the previous section, we denote the global {\it sizes} and {\it energies} (with respect to $\PP'$) by $S_j,E_j$, $j=1,2,3$.

\begin{proposition} \label{prop:estimation} For $\PP'$ a sub-collection of tri-tiles, we have: for every $\theta_1,\theta_2,\theta_3\in[0,1]$ with $\theta_1+\theta_2+\theta_3=1$ then
\begin{align*}
 \Lambda_{\PP'}(f,g,\hh) & \lesssim \ssize_{\PP'}({\bf 1}_G)^{1/r} \cdot S_1^{2\alpha \theta_1} E_1^{1-2\alpha \theta_1} \cdot S_2^{2\alpha \theta_2} E_2^{2\alpha(1-\theta_2)} \cdot S_3^{2\alpha \theta_3} E_3^{1-2\alpha \theta_3} \\
 & + \ssize_{\PP'}({\bf 1}_F)^{1/r} \cdot S_1^{2\alpha \theta_1} E_1^{2\alpha(1-\theta_1)} \cdot S_2^{2\alpha \theta_2} E_2^{1-2\alpha \theta_2} \cdot S_3^{2\alpha \theta_3} E_3^{1-2\alpha \theta_3}.
 \end{align*}
\end{proposition}

\begin{proof} It is an application of the decomposition -- Proposition \ref{prop:decomposition}:
$$ \Lambda_{\PP'}(f,g,\hh) = \sum_{n\in \mathbb{Z}} \Lambda_{\QQ_n}(f,g,\hh)$$
and for every $n\in {\mathbb Z}$
$$ \Lambda_{\QQ_n}(f,g,\hh) = \sum_{\RR \in \FF_n^1} \Lambda_{\RR}(f,g,\hh) + \sum_{\CC\in \FF_n^2} \Lambda_{\CC}(f,g,\hh).$$
For every column $\CC\in \FF_n^2$, Proposition \ref{prop:tree} gives
$$ \Lambda_{\CC}(f,g,\hh) \lesssim \size_{\QQ_n}(f) \cdot \size_{\QQ_n}(g)^{2\alpha} \cdot \ssize_{\PP'}({\bf 1}_G)^{1/r} \cdot \size_{\QQ_n}(\hh) \cdot |I_{\CC}|$$
and similarly for every row $\RR\in\FF_n^1$
$$ \Lambda_{\RR}(f,g,\hh) \lesssim \ssize_{\PP'}({\bf 1}_F)^{1/r} \cdot \size_{\QQ_n}(f)^{2\alpha} \cdot \size_{\QQ_n}(g) \cdot \size_{\QQ_n}(\hh) \cdot |I_{\RR}|.$$
Here we used (as in \cite[Proposition 2.7]{BB}) that by orthogonality along a column or a row, we have
$$ \frac{1}{|I_{\CC}|} \sum_{s\in\CC} |\langle g,\phi_{s_2}\rangle|^2 \lesssim \frac{1}{|I_{\CC}|}\int |g(x)|^2 \chi_{I_{\CC}}(x) dx \lesssim \ssize_{\PP'}({\bf 1}_{G})$$
and similarly
$$ \frac{1}{|I_{\RR}|} \sum_{s\in\RR} |\langle f,\phi_{s_1}\rangle|^2 \lesssim \frac{1}{|I_{\RR}|}\int |f(x)|^2 \chi_{I_{\RR}}(x) dx \lesssim \ssize_{\PP'}({\bf 1}_{F}).$$

So it suffices now to estimate the two components
$$ (I):= \sum_{n\in{\mathbb Z}} \sum_{\CC\in \FF_n^2} \size_{\QQ_n}(f) \cdot \size_{\QQ_n}(g)^{2\alpha} \cdot \size_{\QQ_n}(\hh) \cdot |I_{\CC}| $$
and
$$ (II):= \sum_{n\in{\mathbb Z}} \sum_{\RR\in \FF_n^1} \size_{\QQ_n}(f)^{2\alpha} \cdot \size_{\QQ_n}(g) \cdot \size_{\QQ_n}(\hh) \cdot |I_{\RR}|. $$

Let us focus on $(I)$, since the second one is symmetric. By the property given by the ``selection algorythms'' in Proposition \ref{prop:decomposition}, one has
\begin{align*}
(I) & \lesssim  \sum_{n\in{\mathbb Z}} \sum_{\CC\in \FF_n^2} \min(2^{-n}E_1,S_1) \cdot \min(2^{-n}E_2,S_2)^{2\alpha} \cdot \min(2^{-n}E_3,S_3) \cdot |I_{\CC}| \\
& \lesssim \sum_{n\in{\mathbb Z}}  2^{2n} \cdot \min(2^{-n}E_1,S_1) \cdot \min(2^{-n}E_2,S_2)^{2\alpha} \cdot \min(2^{-n}E_3,S_3) \\
& \lesssim E_1 E_2^{2\alpha} E_3 \cdot \sum_{n\in{\mathbb Z}}  2^{2n} \cdot \min(2^{-n},\frac{S_1}{E_1}) \cdot \min(2^{-n},\frac{S_2}{E_2})^{2\alpha} \cdot \min(2^{-n},\frac{S_3}{E_3}) \\
\end{align*}
and the summation is under the constraint
$$ 2^{-n} E_2 \leq 2 S_2 \qquad \textrm{or} \qquad 2^{-n} E_3 \leq 2 S_3$$
which can be written as
$$ 2^{-n}\leq 2\max( \frac{S_2}{E_2}, \frac{S_3}{E_3}).$$ 

To compute the geometric sum in $n$, we need to compare the three quantities $\frac{S_1}{E_1}$, $\frac{S_2}{E_2}$ and $\frac{S_3}{E_3}$. 
Aiming that, let us work under the assumption
\begin{equation} \frac{S_1}{E_1} \leq \frac{S_2}{E_2} \leq \frac{S_3}{E_3} \label{eq:ass}
\end{equation}
and we let the reader to check the other situations (which can be dealt with in a very similar way).
We will consider the sum in $(I)$, split into several components (still denoted by $(I)$):

\begin{itemize}
\item when $2^{-n} \leq \frac{S_1}{E_1}$, then
\begin{align*}
 (I) & \lesssim E_1 E_2^{2\alpha} E_3 \cdot \sum_{n}  2^{2n} \cdot 2^{-n} \cdot 2^{-2\alpha n} \cdot 2^{-n} \lesssim E_1 E_2^{2\alpha} E_3 \cdot \sum_{n}  2^{-2\alpha n} \\
 & \lesssim E_1 E_2^{2\alpha} E_3 \cdot \big(\frac{S_1}{E_1}\big)^{2\alpha} \\
 & \lesssim E_1 E_2^{2\alpha} E_3 \cdot \big(\frac{S_1}{E_1}\big)^{2\alpha\theta_1} \cdot \big(\frac{S_2}{E_2}\big)^{2\alpha \theta_2} \cdot \big(\frac{S_3}{E_3}\big)^{2\alpha \theta_3} \\
 & \lesssim S_1^{2\alpha \theta_1} E_1^{1-2\alpha \theta_1} \cdot S_2^{2\alpha \theta_2} E_2^{2\alpha-2\alpha \theta_2} \cdot S_3^{2\alpha \theta_3} E_3^{1-2\alpha \theta_3}.
\end{align*}

\item when $\frac{S_1}{E_1} \leq 2^{-n} \leq \frac{S_2}{E_2}$, then
\begin{align*}
 (I) & \lesssim E_1 E_2^{2\alpha} E_3 \cdot \sum_{n}  2^{2n} \cdot \frac{S_1}{E_1} \cdot 2^{-2\alpha n} \cdot 2^{-n} \lesssim S_1 E_2^{2\alpha} E_3 \cdot \sum_{n}  2^{(1-2\alpha) n} \\
 & \lesssim S_1 E_2^{2\alpha} E_3 \cdot \big(\frac{E_1}{S_1}\big)^{(1-2\alpha)} 
\end{align*}
since $\alpha\in(0,\frac{1}{2})$. So
\begin{align*} 
(I) & \lesssim E_1 E_2^{2\alpha} E_3 \cdot \big(\frac{S_1}{E_1}\big)^{2\alpha}
\end{align*}
and we recover the same estimates as previously.

\item when $\frac{S_2}{E_2} \leq 2^{-n} \leq 2\frac{S_3}{E_3}$, then
\begin{align*}
 (I) & \lesssim E_1 E_2^{2\alpha} E_3 \cdot \sum_{n}  2^{2n} \cdot \frac{S_1}{E_1} \cdot \big(\frac{S_2}{E_2}\big)^{2\alpha} \cdot 2^{-n} \lesssim S_1 S_2^{2\alpha} E_3 \cdot \sum_{n}  2^{n} \\
 & \lesssim S_1 S_2^{2\alpha} E_3 \cdot \frac{E_2}{S_2} \lesssim S_1 S_2^{2\alpha} E_3 \cdot \big(\frac{E_2}{S_2}\big)^{2\alpha}\big(\frac{E_2}{S_2}\big)^{1-2\alpha}  \\
  & \lesssim S_1 E_2^{2\alpha} E_3 \cdot \big(\frac{E_1}{S_1}\big)^{1-2\alpha} \lesssim E_1 E_2^{2\alpha} E_3 \cdot \big(\frac{S_1}{E_1}\big)^{2\alpha}
   \end{align*}
and we recover again the same estimate as previously.
\end{itemize}

This completes the proof of the estimate for $(I)$ under \eqref{eq:ass}. The other situations (permutation of the quantities in \eqref{eq:ass}) can be studied similarly, as well as the second term $(II)$ which is symmetric.
\end{proof}

From that estimate, as explained in \cite[Sections 5 and 6]{BB}, in order to get the widest (from such estimates) range, it is necessary to go through an extra localization step. This extra step, detailed in \cite{BB} could nowadays be explained through the {\it sparse domination} point of view (see \cite{BeneaMuscaluSparse}). It was introduced in \cite{BeneaMuscalu} when it has been observed that up to a localization, the {\it energy} could be 'transformed' into a power of {\it size}. We donot detail and we refer the reader to \cite{BB}, we will only sketch the argument to track the exponents.

\bigskip

\begin{theorem} Let $r>2$ be fixed arbitrarily close to $2$ and denote $\alpha:=\frac{1}{2}-\frac{1}{r}>0$. Then for all exponents $p_1,p_2 \in(1,\infty)$ and $p\in(\frac{1}{2},\infty)$ such that $\frac{1}{p}=\frac{1}{p_1}+\frac{1}{p_2}$ and 
$$ \frac{1}{p_1},\frac{1}{p_2},1-\frac{1}{p} < \frac{1}{2}+2\alpha^2$$ the trilinear form $\Lambda_{\PP}$ satisfies the restricted weak-type estimates $(\frac{1}{p_1},\frac{1}{p_2},1-\frac{1}{p})$ with a control (uniform with respect to $\PP$)
$$ \|\Lambda_{\PP}\|_{[\frac{1}{p_1},\frac{1}{p_2},1-\frac{1}{p}]} \lesssim (\sharp {\mathcal S})^{\frac{3}{2} \alpha}.$$
\end{theorem}

\begin{proof}
So we fix arbitrary measurable subsets $F,G,H$ (of finite measure) and functions $f\in X(F)$, $g\in X(G)$ and by homogeneity one can assume $|H|=1$. Then we define the 'usual' exceptional set
$$ \Omega:=\{x,\ M({\bf 1}_{F})(x) \geq c_0 |F|\} \cup \{x,\ M({\bf 1}_{G})(x)> c_0 |G|\}$$
for a numerical constant $c_0$, sufficient large such that $|\Omega|\leq \frac{1}{2}$ and then we set $H':=H\setminus \Omega$ a major subset of $H$.
Then we consider also functions $\hh\in X(H')$ and we decompose the initial collection $\PP:=\bigcup_{d\geq 0} \PP_d$ defined as
$$ \PP_d:=\{s\in \PP, 2^d \leq 1+\frac{d(I_s,\Omega^c)}{|I_s|}\leq 2^{d+1} \}.$$

We refer the reader to \cite[section 6]{BB}, to the following localization step. For integers $n_1,n_2,n_3\geq 0$, there exist $({\mathcal J}_i^{n_i})_{i=1,2,3}$ three collections of disjoint dyadic intervals and for each of these intervals $I\in {\mathcal J}_i^{n_i}$, there is a corresponding collection $\PP_d(I)$ such that
\begin{itemize}
\item[(a)] Control of the averages:
$$ 2^{-n_1} \leq \frac{1}{|I|}\int {\bf 1}_{F}(x) \chi_{I}(x) dx \leq 2^{-n_1}$$
and similarly for $n_2$ with $G$ and $n_3$ with $H'$;
\item[(b)] Control of the local {\it sizes}: for every $I\in {\mathcal J}_1^{n_1}$
$$ \size_{\PP_d(I)}(f) + \ssize_{\PP_d(I)}({\bf 1}_F) \leq 2^{-n_1}$$
and similarly for $n_2$ with $g,G$ and  for $n_3$ and every $I\in {\mathcal J}_3^{n_3}$
$$ \size_{\PP_d(I)}(\hh) + \ssize_{\PP_d(I)}({\bf 1}_{H'})^{1/r'} \leq 2^{-n_3/r'};$$
\item[(c)] Partitions of $\PP_d$: for $i=1,2,3$
$$ \PP_d= \bigcup_{n_i\geq 0} \bigcup_{I\in{\mathcal J}_i^{n_i}} \PP_d(I).$$
\end{itemize}

The fact that we have
$$ \size_{\PP_d(I)}(\hh)\leq 2^{-n_3/r'} \qquad \textrm{and} \qquad \ssize_{\PP_d(I)}({\bf 1}_{H'}) \leq 2^{-n_3}$$
comes from the fact that $\size_{\PP_d(I)}(\hh)$ can be estimated by maximal $L^{r'}$-averages (and not maximal $L^1$-averages as for $f,g$) -- see Proposition \ref{prop:size}.

Then one has
$$ \Lambda_{\PP_d}(f,g,\hh) = \sum_{n_1,n_2,n_3} \sum_{I_1\in{\mathcal J}_1^{n_1} \atop {I_2\in{\mathcal J}_2^{n_2} \atop I_3\in{\mathcal J}_3^{n_3}}} \Lambda_{\PP_d(I_1,I_2,I_3)}(f,g,\hh)$$
where
$$\PP_d(I_1,I_2,I_3)=\PP_d(I_1) \cap \PP_d(I_2) \cap \PP_d(I_3)$$
and the integers $n_i$ are such that (because they have to be controlled by the {\it size} of $\PP_d$) 
$$ 2^{-n_1}\lesssim 2^d |F|, \qquad 2^{-n_2} \lesssim 2^d|G|, \qquad 2^{-n_3} \lesssim 2^{-Md} \cdot (\sharp {\mathcal S})^{\alpha}$$
where $M$ is as large as we want. Indeed it is standard that because of the definition of the exceptional subset and of $\PP_d$ then
$$ \size_{\PP_d}(f) \leq 2^{d} |F|, \qquad \size_{\PP_d}(g) \leq 2^{d} |G|, \qquad \size_{\PP_d}(\hh) \leq 2^{-Md} \cdot (\sharp {\mathcal S})^{\alpha}.$$ 

For every $I_1,I_2,I_3$, we apply Proposition \ref{prop:estimation} (with the fact that now the {\it energies} which are controlled by local$-L^2$ averages are then bounded by $2^{-n_i/2}|I_1 \cap I_2 \cap I_3|^{1/2}$ for $i=1,2,3$):
\begin{align*}
 \Lambda_{\PP_d(I_1,I_2,I_3)}(f,g,\hh) & \lesssim 2^{-n_2/r} \cdot 2^{-2\alpha \theta_1 n_1} 2^{-(1-2\alpha \theta_1) n_1/2} |I_1 \cap I_2 \cap I_3|^{(1-2\alpha \theta_1)/2} \\ 
& \cdot 2^{-2\alpha \theta_2 n_2} 2^{-2\alpha(1-\theta_2) n_2/2} |I_1 \cap I_2 \cap I_3|^{2\alpha(1- \theta_2)/2} \\
&  \cdot 2^{-2\alpha \theta_3 n_3/r'} (\sharp {\mathcal S})^{\alpha(1-2\alpha \theta_3)} 2^{-(1-2\alpha \theta_3) n_3/2} |I_1 \cap I_2 \cap I_3|^{(1-2\alpha \theta_3)/2} \\
 & + 2^{-n_1/r} \cdot 2^{-2\alpha \theta_1 n_1} 2^{-2\alpha(1-\theta_1) n_1/2} |I_1 \cap I_2 \cap I_3|^{2\alpha(1-\theta_1)/2} \\ 
& \cdot 2^{-2\alpha \theta_2 n_2} 2^{-(1-2\alpha\theta_2) n_2/2} |I_1 \cap I_2 \cap I_3|^{(1- 2\alpha \theta_2)/2} \\
&  \cdot 2^{-2\alpha \theta_3 n_3/r'} (\sharp {\mathcal S})^{\alpha(1-2\alpha \theta_3)} 2^{-(1-2\alpha \theta_3) n_3/2} |I_1 \cap I_2 \cap I_3|^{(1-2\alpha \theta_3)/2}.
 \end{align*}
That gives us (since $\theta_1+\theta_2+\theta_3=1$)
\begin{align*}
 \Lambda_{\PP_d(I_1,I_2,I_3)}(f,g,\hh) & \lesssim (\sharp {\mathcal S})^{\alpha(1-2\alpha \theta_3)} \cdot 2^{-(\frac{1}{2}+\alpha \theta_1) n_1}  \cdot 2^{-(\frac{1}{2}+\alpha \theta_2) n_2} \cdot 2^{-(\frac{1}{2}+2\alpha^2 \theta_3) n_3} \cdot
 |I_1 \cap I_2 \cap I_3|.
 \end{align*}
Then we will have to sum over $n_1,n_2,n_3$ such that 
\begin{equation} 2^{-n_1} \lesssim \min(1,2^d|F|), \qquad 2^{-n_2}\lesssim \min(1,2^d|G|), \qquad 2^{-n_3}\lesssim 2^{-Md} \cdot (\sharp {\mathcal S})^{\alpha} \label{eq:ni} \end{equation}
because all the {\it sizes} are also bounded by a constant and so we get that as soon as 
$0\leq \nu_1 \leq \frac{1}{2}+\alpha \theta_1$, $0\leq \nu_2 \leq \frac{1}{2}+\alpha \theta_2$ and $0\leq \tilde \nu_3 \leq \frac{1}{2}+2\alpha^2 \theta_3$ we have
\begin{align*}
 \Lambda_{\PP_d(I_1,I_2,I_3)}(f,g,\hh)  \lesssim (\sharp {\mathcal S})^{\alpha(1-2\alpha \theta_3)} \cdot 2^{-\nu_1 n_1}  \cdot 2^{-\nu_2 n_2} \cdot 2^{- \tilde\nu_3 n_3} \cdot
 |I_1 \cap I_2 \cap I_3|.
 \end{align*}
Then we have to sum over the intervals  
$I_i\in {\mathcal J}_i^{n_i}$. Due to Property (a) above of the selected intervals ${\mathcal J}_i^{n_i}$, we have that
$$ \sum_{I_1\in{\mathcal J}_1^{n_1}} |I_1| \lesssim 2^{n_1} |F|$$
and similarly for $i=2,3$, due to the disjointness. And indeed we also have
$$ \sum_{I_1\in{\mathcal J}_1^{n_1} \atop {I_2\in{\mathcal J}_2^{n_2} \atop I_3\in{\mathcal J}_3^{n_3}}} |I_1\cap I_2 \cap I_3| \lesssim \min\Big(\sum_{I_1\in{\mathcal J}_1^{n_1}} |I_1|,\sum_{I_2\in{\mathcal J}_2^{n_2}} |I_2|,\sum_{I_3\in{\mathcal J}_3^{n_3}} |I_3|\Big). $$
So for $\gamma_i\in[0,1]$ with $\gamma_1+\gamma_2+\gamma_3=1$, one has
$$ \sum_{I_1\in{\mathcal J}_1^{n_1} \atop {I_2\in{\mathcal J}_2^{n_2} \atop I_3\in{\mathcal J}_3^{n_3}}} |I_1\cap I_2 \cap I_3| \lesssim (2^{n_1}|F|)^{\gamma_1} (2^{n_2}|G|)^{\gamma_2} 2^{n_3\gamma_3}. $$
Hence finally,
\begin{align*}
 \sum_{I_1\in{\mathcal J}_1^{n_1} \atop {I_2\in{\mathcal J}_2^{n_2} \atop I_3\in{\mathcal J}_3^{n_3}}} \Lambda_{\PP_d(I_1,I_2,I_3)}(f,g,\hh)  \lesssim (\sharp {\mathcal S})^{\alpha(1-2\alpha \theta_3)} \cdot 2^{(\gamma_1-\nu_1) n_1}  \cdot 2^{(\gamma_2-\nu_2) n_2} \cdot 2^{(\gamma_3- \tilde \nu_3) n_3} \cdot
 |F|^{\gamma_1} |G|^{\gamma_2}.
 \end{align*}
We can then sum over $n_i$ under the condition \eqref{eq:ni} with $\nu_i> \gamma_i$ and in such a case we have
$$ \Lambda_{\PP_d}(f,g,\hh) \lesssim (\sharp {\mathcal S})^{\alpha(1-2\alpha \theta_3+\tilde \nu_3-\gamma_3)} \cdot 2^{-M'd}|F|^{\nu_1} |G|^{\nu_2}$$
for some numerical exponent $M'$ (depending of all the exponents). Since $\alpha\in(0,\frac{1}{2})$ then
$$ 1-2\alpha \theta_3+\tilde \nu_3-\gamma_3 \leq 1-2\alpha \theta_3 + \frac{1}{2}+\alpha \theta_3 \leq \frac{3}{2}$$
and so we deduce that
\begin{equation} \Lambda_{\PP}(f,g,\hh) = \sum_{d\geq 0} \Lambda_{\PP}(f,g,\hh) \lesssim (\sharp {\mathcal S})^{\frac{3}{2}\alpha} \cdot |F|^{\nu_1} |G|^{\nu_2}. \label{eq:td}
\end{equation}
That corresponds to the desired restricted weak-type estimates for the exponents $(\nu_1,\nu_2,\nu_3)$ with $\nu_3:=1-(\nu_1+\nu_2)$.

The constraints (that are needed in the previous computations), are
$$ 0\leq \nu_1\leq \frac{1}{2}+\alpha \theta_1, \qquad 0\leq \nu_2\leq \frac{1}{2}+\alpha \theta_2,$$
and
$$ \gamma_1<\nu_1, \qquad \gamma_2<\nu_2, \qquad \gamma_3<\tilde \nu_3 \leq \frac{1}{2}+2\alpha^2 \theta_3$$
with $\theta_i,\gamma_i\in[0,1)$ and 
$$ \theta_1+\theta_2+\theta_3=\gamma_1+\gamma_2+\gamma_3=1.$$

In particular, these contraints are satisfied if
$$ 0< \nu_1,\nu_2 < \frac{1}{2}+2\alpha^2 \qquad \textrm{and} \qquad \nu_3 < \frac{1}{2}+2\alpha^2.$$
Indeed for such exponents, one has
$$ \frac{\nu_i-\frac{1}{2}}{2\alpha^2} <1 \qquad \textrm{and}  \qquad \sum_{i=1}^3 \frac{\nu_i-\frac{1}{2}}{2\alpha^2} = -\frac{1}{4\alpha^2} <1.$$
So we can find admissible exponents $\theta_i\in(0,1)$ such that for $i=1,2,3$
$$ \nu_i <\frac{1}{2} +2\alpha^2 \theta_i$$
and since $\alpha\in(0,\frac{1}{2})$, 
$$ 0\leq \nu_1\leq \frac{1}{2}+2\alpha^2 \theta_1< \frac{1}{2}+\alpha \theta_1, \qquad 0\leq \nu_2\leq \frac{1}{2}+2\alpha^2 \theta_2 < \frac{1}{2}+\alpha \theta_2.$$
Then by chosing $\tilde \nu_3 \in (\max(\nu_3,0), \frac{1}{2} +2\alpha^2 \theta_3)$, one has $\nu_1,\nu_2,\tilde \nu_3 \in(0,1)$ with 
$\nu_1+\nu_2+\tilde \nu_3 >1$ and so there exists $\gamma_i\in(0,1)$ such that $0<\gamma_1<\nu_1$, $0<\gamma_2<\nu_2$, $0<\gamma_3<\tilde \nu_3$ with $\gamma_1+\gamma_2+\gamma_3=1$. So we have found suitable exponents verifying all the conditions, in order to make the previous computations valid, which concludes the proof of the statement with $\nu_1=\frac{1}{p_1}$, $\nu_2=\frac{1}{p_2}$ and $\nu_3=1-\frac{1}{p}$.
\end{proof}

\begin{remark} If we follow \cite{BeneaMuscaluSparse} which explains how such localization algorythm can give a {\it sparse domination} (indeed for $i=1,2,3$, the collection $\bigcup_{n_i\geq 0} {\mathcal J}_i^{n_i}$ is {\it sparse}), we can then prove that the trilinear form $\Lambda_{\PP}$ satisfies the following {\it sparse domination}: for $f,g,\hh$ there exists a {\it sparse} collection of intervals $(I)_{I}$ such that
$$ |\Lambda(f,g,\hh)| \lesssim (\sharp {\mathcal S})^{\frac{3}{2} \alpha} \cdot \sum_{I} \left(\aver{I} |f|^{r_1}\right)^{\frac{1}{r_1}} \cdot \left(\aver{I} |g|^{r_2}\right)^{\frac{1}{r_2}} \cdot \left(\aver{I} |\hh|^{r_3}\right)^{\frac{1}{r_3}} \cdot |I| $$
for exponents $r_1,r_2,r_3$ corresponding to
$$ \frac{1}{r_1}:=\frac{1+\alpha \theta_1}{2}, \qquad \frac{1}{r_2}:=\frac{1+\alpha \theta_2}{2} \qquad \textrm{and} \qquad \frac{1}{r_3}:=\frac{1}{2}+2\alpha^2 \theta_3.$$
We know that such {\it sparse domination} implies boundedness from $L^{p_1} \times L^{p_2}$ to $L^{p}$ for exponents $p_1,p_2,p$ such that $\frac{1}{p_1}+\frac{1}{p_2}=\frac{1}{p}$ and
$$ r_1<p_1, \qquad r_2<p_2 \qquad \textrm{and} \qquad r_3< p',$$
which allows to recover the exact same range of boundedness (up to the use $2\alpha\leq 1$ for a simplification of the conditionn that we also used in the previous proof).
\end{remark}

\begin{remark} In the final statement, we put a loss of order $(\sharp {\mathcal S})^{\frac{3}{2} \alpha}$ and the $\frac{3}{2}$ comes from \eqref{eq:td}. If we do things a bit more carefully we could reach an exponent $1$.
\end{remark}

\end{document}